\documentclass[twoside,
]{amsart}

  \usepackage[utf8]{inputenc}
  \usepackage{amsmath}
\usepackage{enumerate}
  \usepackage{amssymb}
  \usepackage{amsthm}
  \usepackage{bm}
  \usepackage{graphicx}
  \usepackage{color}
\usepackage{accents}
\usepackage{url}
\usepackage{epigraph}

\newtheorem{Th}{Theorem}
\newtheorem{Prop}{Proposition}
\newtheorem{lem}{Lemma}

\newtheorem{Cor}{Corollary}

\newtheorem*{conj}{Conjecture}
\newtheorem*{thank}{\ \ \ Acknowledgment}

\def\scalar(#1,#2){(#1\mid#2)}

\newcommand{\cp}{\mathcal{P}}

\newcommand{\R}{{\mathbb{R}}}

\newcommand{\T}{{\mathbb{T}}}

\newcommand{\1}{{\mathbb{I}}}

\newcommand{\Lit}{\mathcal{L}}

\newcommand{\tend}[3][]{\xrightarrow[#2\to#3]{#1}}

\newcommand{\ds}{\displaystyle}

\title{On $L^\alpha$-flatness of Erd\H{o}s-Littlewood's polynomials.}

\author{\lowercase{el} H\lowercase{oucein} \lowercase{el} A\lowercase{bdalaoui}$^\star$}
\address{$^\star$ University of Rouen Normandy ,
	Department of Mathematics, LMRS  UMR 60 85 CNRS\\
	Avenue de l'Universit\'e, BP.12
	76801 Saint Etienne du Rouvray - France .}
\email{elhoucein.elabdalaoui@univ-rouen.fr}

\urladdr{http://www.univ-rouen.fr/LMRS/Persopage/Elabdalaoui/}

\date{\today}
\subjclass[2020]{Primary 42A05, 42A55, Secondary 37A05, 37A30}
\dedicatory{}

\keywords{flat polynomials, ultraflat polynomials, Erd\"{o}-Littlewood's problem, 
 Dirichlet kernel, Marcinkiewicz-Zygmund's interpolation inequalities, Simple Lebesgue component spectrum, Banach's problem,  Banach-Rokhlin's problem, weak Rokhlin's problem.
}

\begin{document}

\maketitle

\epigraph{Begin at the beginning, the King said gravely, ``and go on till you come to the end: then stop.''}{Lewis Carroll, \textit{Alice in Wonderland}}

\epigraph{Those who know do not speak; those who speak do not know.}{Laozi (Lao Tzu)$^{1}${\footnote {This Taoist idea can be rephrased à la Erd\"{o}s's as : "Everyone writes, nobody reads."}}}

\begin{abstract}
It is shown that Erd\H{o}s-Littlewood's polynomials are not $L^\alpha$-flat when $\alpha>2$ is an even integer (and hence for any $\alpha \geq 4$). This provides a partial solution to an old problem posed by Littlewood. Consequently, we obtain a positive answer to the analogous Erd\"{o}s-Newman conjecture for polynomials with coefficients $\pm 1$; that is, there is no ultraflat sequence of polynomials from the class of Erd\H{o}s–Littlewood polynomials.

Our proof is short and simple. It relies on the classical lemma for $L^p$ norms of the Dirichlet kernel, the Marcinkiewicz-Zygmund interpolation inequalities, and the $p$-concentration theorem due to A. Bonami and S. R\'ev\'esz.
\end{abstract}

\section{Introduction}
Let $\Lit$ be the class of analytic trigonometric polynomials of the form 
$$P_q(\theta)=\frac{1}{\sqrt{q}}\sum_{j=0}^{q-1}\epsilon_k e^{i k\theta},$$
where $\epsilon_k =\pm 1$. This class is said to be the class of Littlewood. It may be said also the class of Erd\"{o}s-Littlewood (see \cite{L1,L2,E}).\\

Here, we are interest on the behavior of those polynomials.  Precisely on the $L^\alpha$-flatness, $\alpha \geq 0$ of a sequence $(P_n)$ from $\Lit$ . As mentioned above, this problem arises in pure mathematics \cite{L1, L2, E}, and it turns out that it also arises in several engineering problems \cite{S,XS, O}.\\

A sequence $(P_q)$ of analytic trigonometric polynomials of class $\Lit$ is said to be $L^\alpha$-$c$-flat if $|P_q|$ converge to  a constant $c$ with respect to $L^\alpha$-norm. Formally, 
\begin{eqnarray}\label{Def}
\Big\|\big|P_q(z)\big|-c\Big\|_\alpha \tend{q}{+\infty}0. 
\end{eqnarray}
If $c=1$ the sequence $(P_q)$ is said to be  $L^\alpha$-flat.\\
As it is customary, let us denote the torus by $\mathbb{T} = \{z \in \mathbb{C}: |z|=1\}$, which can be identified with $[0, 1)$, $[-\frac{1}{2}, \frac{1}{2})$, $[0, 2\pi)$, or $[-\pi, \pi)$. Obviousily, $L^\alpha$-flatness implies $L^\beta$-flatness, for any $\beta \leq \alpha$, since for any polynomials $P$, $\big\|P\big\|_\beta \leq \big\|P\big\|_\alpha.$

In this note, we present a straightforward and short proof showing that Erd\"{o}s-Littlewood polynomials are not $L^\alpha$-flat when $\alpha$ is an even integer striclty greater than $2$. This provides a positive answer to the Erdős conjecture (Problem 22 in \cite{E}) for the class $\Lit$,  extending the result in \cite{AE} and also supporting the numerical computations by A. Odlyzko \cite{O}.

It turns out that J. Bourgain and M. Guenais established a connection between the $L^1$-flatness of Erd\H{o}s-Littlewood polynomials and a long-standing problem in the spectral theory of dynamical systems attributed to Banach and Rokhlin \cite{B,G}. This problem concerns the existence of a measure-preserving ergodic transformation on a probability measure space with a simple Lebesgue spectrum. Let us recall that the original Banach problem asked whether there exists a map acting on $\mathbb{R}$ with a simple Lebesgue spectrum. J. Bourgain linked this problem to the problem of $L^1$-flatness in the class of polynomials nodays known as Bourgain-Newman polynomials \cite{Abd-Nad, abd-nad2, B}. 

In \cite{Ab}, the author provided a positive answer to the original Banach problem by constructing an ergodic transformation acting on an infinite measure space with a simple Lebesgue spectrum (see also \cite{Ab2} for a simpler proof). However, the author emphasized therein that the $L^1$-flatness of Erd\H{o}s-Littlewood polynomials remains an open question.

It should be noted that our work does not provide any insight into the $L^\alpha$-flatness problem for Erd\H{o}s--Littlewood polynomials when $\alpha < 2$. Furthermore, we do not believe that our method is applicable in this case.

\section{Main result and its proof.}

In this section, we will state and prove our main result.

\begin{Th}\label{main1}[Theorem of El Roc de Sant Gaiet\`{a} \footnote{ Bachelard points out that places have a profound effect on our imagination and can inspire ideas and works \cite{B}. So, it can be suggested to name theorems after specific places.!}] There is no sequence from the class $\Lit$ which is $L^{2p}$-flat, for any  positive integer $p>1$.
\end{Th}
\noindent{}Consequently, we have, 
\begin{Cor}There is no sequence from the class $\Lit$ which is $L^{\alpha}$-flat, for any $\alpha \geq 4$.
\end{Cor}
\begin{proof} Assume that there is a sequence from the class $\Lit$ which is $L^{\alpha}$-flat, for $\alpha \geq 4$. There this sequence is $L^{2 p}$-flat, where $p=\lfloor \frac{\alpha}{2} \rfloor$, which contradicts our main result. 
\end{proof}
We further have
\begin{Cor}The Erd\"{o}s conjecture on the existence of ultraflat from the class $\Lit$ is true, that is, there is no ultraflat sequence of polynomials $(P_q) \subset \Lit$.
\end{Cor}
\begin{proof} Assume by contradiction that such a sequence exists. Then this sequence is $L^\alpha$-flat for any $\alpha>0$. This contradicts our main result and the proof is complete.
\end{proof}
For the proof of Theorem \ref{main1}, we need the following classical lemma on the  computation of $L^p$ norms of Dirichlet Kernel. Its proof can be founded here \cite{AAJRS}.
\begin{lem}[$L^p$-norm of Dirichlet Kernel]\label{BZ}Let $N$ be a positive integer and $$\ds D_N(z)=\sum_{j=0}^{N-1}z^j.$$ Then, for any $p>1$,
$$\|D_N\|_p^p=\delta_p N^{p-1}+R_p(N^{p-1}), \textrm{~~as~~} N \longrightarrow +\infty,$$
where 
$$\delta_p=\frac{2}{\pi}\int_{0}^{\infty}\Big|\frac{\sin(x)}{x}\Big|^p dx, $$
and 
$$R_p(N)=\begin{cases}
O_p(N^{p-3}) &\textrm{~~if~~} p>3,\\
O_p(\ln(N))   &\textrm{~~if~~} p=3,\\
O_p(1)  &\textrm{~~if~~} 1 \leq p<3,\\
\end{cases}
$$
\end{lem}
We need also the following observation form \cite{Abd-nad3} (see eq. (2.2) and (2.3))
\begin{align}\label{Obs}
P_q(z)&=\frac{1}{\sqrt{q}}D_q(z)-\frac{2}{\sqrt{q}}\sum_{j=0}^{N-1}\eta_j' z^j\\
&=\frac{2}{\sqrt{q}}\sum_{j=0}^{N-1}\eta_j z^j-\frac{1}{\sqrt{q}}D_q(z),\\ &\textrm{~~where~~} e^{i \theta}=z, \theta \in \R, \textrm{~~and~~} \eta_j,\eta_j' \in \{0,1\}, j=0,\cdots, N-1.
\end{align} 
We recall also the following lemma from  \cite{Ab} which is related to the so-called $L^4$-strategy due D.J. Newman \&  Byrnes \cite{N-B} .
\begin{lem}[$L^4$-norm of Bourgain-Newmann polynomials]\label{BN} 
Let $q$ be a positive integer and $$\ds Q_q(z)=\sum_{j=0}^{q-1}\eta_je^{2\pi jx},$$ where 
$x \in [0,1),\eta_j \in \{0,1\}, j=0,\cdots, q-1.$ Then,
$$\liminf_{q\rightarrow +\infty} \frac{\Big\|Q_q\Big\|_4 }{\Big\|Q_q\Big\|_2}\geq 2.$$ 
\end{lem}
From Lemma \ref{BN}, it can be shown that the Newmann-Bourgain polynomials are not $L^\alpha$-flat for $\alpha \geq 4$ (see \cite{Abd-Nad-H}). In a forthcoming paper, we will use it again to prove that Ben Green's polynomials are not $L^\alpha$-flat for $\alpha>0.$\\

\begin{Prop}\label{no-flat-density} Let $(Q_q)$ be a sequence of Bourgain-Newmann polynomials and assume that the density of $1$ is positive. Then, for any $\alpha>2$,
$$\frac{\Big\|Q_q\Big\|_\alpha}{\Big\|Q_q\Big\|_2}\tend{q}{+\infty}+\infty.$$
\end{Prop}

 In addition, we require the following lemma, which corresponds to the  Marcinkiewicz-Zygmund interpolation inequalities \cite[ Theorem 2.7, Chap. X,  Vol. II, p. 30]{Zyg}.
\begin{lem}\label{MZ} [Marcinkiewicz-Zygmund interpolation inequalities ]Let $\alpha \in [1,+\infty]$. Then, for any analytic polynomial $P$ of degree at most $n-1$, there are two positive constants $A', A'_{\alpha}$ such that
\begin{eqnarray}\label{eq:mz}
&&\Big(\frac{1}{n}\sum_{j=0}^{n-1}\big|P(\xi_{j,n})\big|^{\alpha}\Big)^{\frac{1}{\alpha}}
\leq A'. 
\Big\|P(z)\Big\|_{\alpha}  \textrm{~~if~~} \alpha \in [1,+\infty],\\
&&\Big\|P(z)\Big\|_{\alpha} \leq \Big(\frac{A'_{\alpha}}{n}\sum_{j=0}^{n-1}\big|P(\xi_{n,j})\big|^{\alpha}\Big)^{\frac{1}{\alpha}}, \textrm{~~if~~} \alpha \in ]1,+\infty[
\end{eqnarray}
where $(\xi_{n,j})$ are the $n$-th root of unity given by 
$$\xi_{n,j}=e^{2\pi i\frac{j}{n}},~~~j=0,\cdots,n-1.$$
\end{lem}

According to Zygmund \cite[vol II, Chap X, p. 31]{Zyg}, the constants are defined as $A' = 2A + 1$ and $A'_\alpha = 2A_\alpha + 1$, where $A = \sup_{\alpha>1}(\pi \alpha + 1)^{\frac{1}{\alpha}}$ and $A_\alpha$ is a similar constant obtained for the Marcinkiewicz-Zygmund inequalities for trigonometric polynomials of degree at most $n$ (see also \cite{MZ}). Moreover, the constant $A_\alpha$ does not exceed a fixed multiple, independent of $\alpha$, of the constant in the Riesz Theorem for the conjugate function \cite[Vol I, Chap VII, p. 253]{Zyg}.

\noindent{}The sharp constant in that theorem are 
$$
A_\alpha =\begin{cases}
\tan\big(\frac{\pi}{2\alpha}\big) &\textrm{~~if~~} 1<\alpha \leq 2,\\ 
\cot\big(\frac{\pi}{2\alpha}\big) &\textrm{~~if~~} 2\leq\alpha <+\infty 
\end{cases}
$$
 This later result is due to Pichorides. We refer to \cite{Es} , for a survey on the subject.

Let us also recall the following lemma from \cite{Ab}. For the sake of completeness, we will provide its proof. We believe that this lemma will find application in further research.

\begin{lem}[Flatness implies zero density]\label{Den}Let $q$ be a positive integer, $c$ a positive number, $\alpha >2$ and $$\ds Q_q(x)=\sum_{j=0}^{q-1}\eta_je^{2\pi jx},$$ where 
$x \in [0,1),\eta_j \in \{0,1\}, j=0,\cdots, q-1.$ Assume that the sequence
$\ds \Big(\frac{1}{\sqrt{q}}.Q_q(x)\Big)$ is $L^\alpha$-$c$-flat. Then, the density of the set $S_q=\Big\{0 \leq j \leq q-1~~:~~\eta_j=1\Big\} $ vanish as $q \rightarrow +\infty$, that is,
 
$$\frac{\big|S_q\big|}{q}\tend{q}{+\infty}0.$$
\end{lem}
\begin{proof} For $\alpha>1$, by Lemma \ref{MZ}, we have

\begin{align}
\Big\|\frac{1}{\sqrt{q}}.Q_q(x)\Big\|_{\alpha}^{\alpha} &\geq A'_\alpha.
\frac{1}{q}. \Big|\frac{Q_q(1)}{\sqrt{q}}\Big|^\alpha=
A'_\alpha \frac{2^\alpha}{q^{1+\frac{\alpha}{2}}}.q^\alpha.\Big(\frac{\big|S_q\big|}{q}\Big)^\alpha\\
&\geq A'_\alpha \frac{2^\alpha}{q^{1-\frac{\alpha}{2}}}.\Big(\frac{\big|S_q\big|}{q}\Big)^\alpha
\end{align}
But, by our assumption $\alpha>2$ and  the sequence
$\ds \Big(\frac{1}{\sqrt{q}}.Q_q(x)\Big)$ is $L^\alpha$-$c$-flat. Therefore
$$\Big\|\frac{1}{\sqrt{q}}.Q_q(x)\Big\|_{\alpha}^{\alpha} \tend{q}{+\infty}c^\alpha,$$
and 
$$q^{1-\frac{\alpha}{2}}\tend{q}{+\infty}0.$$
We thus conclude
$$\frac{\big|S_q\big|}{q}\tend{q}{+\infty}0,$$
and the lemma has been proven.
\end{proof}

Under certain restricted conditions, the preceding proof enables us to improve the result in \cite{Abd-Nad-H} as follows.

\begin{Prop}\label{density-flat} Let $(Q_q)$ be a sequence of Bourgain-Newmann polynomials and assume that the density of $1$ is positive. Then, for any $\alpha>2$,
$$\frac{\Big\|Q_q\Big\|_\alpha}{\Big\|Q_q\Big\|_2}\tend{q}{+\infty}+\infty.$$
\end{Prop}
\begin{proof}
 Let  $\ds Q_q(z)=\sum_{j=0}^{q-1}\eta_jz^j,$ where 
$z \in \T,\eta_j \in \{0,1\}, j=0,\cdots, q-1,$ and put 
$$S_q=\{ 0 \leq j \leq q-1~~:~~\eta_j=1\}.$$ Then, by our assumption 
$$\frac{|S_q|}{q}\tend{q}{+\infty}d>0.$$
Applying Lemma \ref{MZ}, it follows
\begin{align}
\Big\|\frac{1}{\sqrt{|S_q|}}.Q_q(z)\Big\|_{\alpha}^{\alpha} \geq A'_\alpha.
\frac{1}{q}. \Big|\frac{Q_q(1)}{\sqrt{|S_q|}}\Big|^\alpha
\end{align}
\noindent{}But $|S_q| \sim d.q.$ Therefore,
\begin{align}
\Big\|\frac{1}{\sqrt{|S_q|}}.Q_q(z)\Big\|_{\alpha}^{\alpha} &\geq A'_\alpha.
\frac{1}{q}. \Big(\sqrt{|S_q|}\Big)^\alpha\\
&\gtrsim c_{\alpha} q^{\frac{\alpha}{2}-1} \tend{q}{+\infty}+\infty,
\end{align}

\noindent{}since $\alpha>2$. The proof of the proposition is complete.
\end{proof}

We further need a crucial lemma from \cite{Bon1} which corresponds to one of the main results therein. Before stating the lemma, we introduce the following definition

Let $\cp$ be the set of polynomials with coefficients $0$ or $1$. ( called also idempotents). For each $p >0$, define
$C_p$ as the largest number such that for every set $E$, $E \subset \T$  with $|E| > 0$, the inequality
$$\sup_{P \in \cp} \frac{\ds \int_E |P(z|^pdz}{\ds \int |P(z|^p dz} \geq C_p.$$
holds. $|E|$ denotes the Lebesgue measure of $E$ and the definition of $C_p$ is extended to the limit case $p = \infty$ in the usual way.

Following \cite{Bon2}, For $p > 0$, we say that there is $p$-concentration if there
exists a constant $c > 0$ so that for any symmetric non empty open set
$E$ one can find an idempotent $P \in \cp$ with
\begin{eqnarray}
\int_E |P(z)|^p dz \geq c \int |P(z)|^p dz.
\end{eqnarray}
Moreover, $c_p$ will denote the supremum of all such constants $c$. Correspondingly, cp is called the level of p-concentration. If $c_p = 1$, we say that there is full $p$-concentration. In \cite{Bon2}, the authors proved the following.

\begin{lem} \label{Bo-Re}
For all $0 < p < \infty$ we have $p$-concentration. Moreover,
if $p$ is not an even integer, then we have full concentration, i.e. $c_p = 1.$
When considering even integers, we have $c_2=\sup_{x \geq 0}\frac{\sin(x)^2}{\pi x}$ , then $0.495 <
c_4 \leq 1/2$, then for all other even integers $0.483 < c_{2k} \leq 1/2.$
\end{lem}

The previous lemma addressed the $p$-concentration inequalities (see \cite{Bon1,Bon2} for details). 
The constant $c_2$ is due to D\'echamps-Gondim, Lust-Piquard and Queffélec \cite{DLPQ}. It is also related to the classical Weiner-Shapiro inequality \cite{Sha} (see also \cite{Bon3}).\\ 

We now proceed with the proof of our main theorem (Theorem \ref{main1}). 

\begin{proof} By \eqref{Obs}, we have 
\begin{align}
P_q(z)=\frac{1}{\sqrt{q}}\Big(2Q_q(z)-D_q(z)\Big),
\end{align}
where $\ds Q_q(z)=\sum_{j=0}^{q-1}\eta_jz^j$ and
$\eta_j \in \{0,1\}, j=0,\cdots, q-1.$ 
It follows that for any $q$-th root of unity $\xi_{j,q}, j=1,\cdots,q-1$, we  can write
\begin{align}\label{Unity}
P_q(\xi_{j,q})=\frac{-2}{\sqrt{q}}.Q_q(\xi_{j,q})
\end{align}
Let us assume that $(P_q)$ is $L^\alpha$-flat for $\alpha>2.$ Then, we have
$$\int \big|\big|P_q(z)\big|-1\big|^\alpha dz \tend{q}{+\infty}0.$$
We can thus extract a subsequence which we still denote by $(P_q)$ such that for almost all $ z \in \T$, we have
$$\big|P_q(z)\big| \tend{q}{+\infty}1.$$
Hence, for almost all $ z \in \T$,
$$ \Big|\frac{2}{\sqrt{q}}Q_q(z)\Big| \tend{q}{+\infty}1.$$
Since, for any $z \neq 1$,
$$\frac{|D_q(z)|}{\sqrt{q}} \tend{q}{+\infty}0.$$
Therefore, by appling Vitali's theorem \cite{Abd-nad3}, for any $\beta<2,$
$$ \Big\|\frac{2}{\sqrt{q}}Q_q(z)-1\Big\|_\beta \tend{q}{+\infty}0.$$
Moroever, by the triangle inequalities combined with Lemma \ref{BZ}, we have
\begin{align}
\Big|\frac{\Big\|\frac{2}{\sqrt{q}}Q_q(z)\Big\|_\alpha}{\Big\|\frac{D_q(z)}{\sqrt{q}}\Big\|_\alpha}-1\Big| \leq \frac{\Big\|P_q\Big\|_\alpha}{\Big\|\frac{D_q(z)}{\sqrt{q}}\Big\|_\alpha} \tend{q}{+\infty}0.
\end{align}
Whence 
\begin{eqnarray}\label{BNequivD}
\frac{\Big\|2.Q_q(z)\Big\|_\alpha}{\Big\|D_q(z)\Big\|_\alpha}  \tend{q}{+\infty}1.
\end{eqnarray}
Now, apply Lemma \ref{MZ} to the analytic polynomial $\frac{2}{\sqrt{q}}.Q_q$ to get 
\begin{align}\label{LittleQ}
\Big\|\frac{2}{\sqrt{q}}Q_q(z)\Big\|_{\alpha}^{\alpha} \leq A'_\alpha 
\Big(\frac{1}{q}.\Big|\frac{2}{\sqrt{q}}Q_q(1)\Big|^\alpha+
\frac{1}{q}\sum_{j=1}^{q-1}\Big|\frac{1}{\sqrt{q}}.Q(\xi_{j,q})\Big|^{\alpha}\Big)
\end{align}
We claim that we can assume without loss of generality that the sequence $$\Big(\frac{1}{q}\sum_{j=1}^{q-1}\Big|\frac{1}{\sqrt{q}}.Q(\xi_{j,q})\Big|^{\alpha}\Big)_{q \geq 1}$$ converge to some constante $C_\alpha$. Indeed, by applying once again Lemma \ref{MZ} to the the analytic polynomial $P_q$, we can write
\begin{align}
A'_\alpha. \frac{1}{q}\sum_{j=0}^{q-1}\Big|P_q(\xi_{j,q})\Big|^{\alpha}
\leq \Big\|P_q\Big\|_\alpha.
\end{align}
But under our assumption, 
for $\alpha >2$, we have
$$\int \big|\big|P_q(z)\big|-1\big|^\alpha dz \tend{q}{+\infty}0.$$
Therefore,
$$\Big\|P_q\Big\|_\alpha \tend{q}{+\infty}1.$$
We thus deduce that the sequence
$$\Big(\frac{1}{q}\sum_{j=0}^{q-1}\big|P(\xi_{j,q})\big|^{\alpha}\Big)_{q \geq 1}$$
is a bounded sequence. So that we can extract a convergent subsequence from it. But, from \eqref{Unity}, we have 
$$
\frac{1}{q}\sum_{j=0}^{q-1}\big|P(\xi_{j,q})\big|^{\alpha}=\frac{1}{q}\sum_{j=1}^{q-1}\Big|\frac{2}{\sqrt{q}}.Q(\xi_{j,q})\Big|^{\alpha}.$$
We can  thus assume without loss of generality that it converge to some positive constant $C_\alpha$. This proved the claim. We still need to estimate 
$$\frac{1}{q}.\Big|\frac{2}{\sqrt{q}}Q_q(1)\Big|$$
For that notice that we have
$$Q_q(1)=\sum_{j=0}^{q-1}\eta_j=|S_q|,$$
and by appealing once again to Lemma \ref{MZ}, it can shown that 
the density of $S_q$ converge to $\frac{1}{2},$
that is, 
\begin{align}\label{half}
\frac{|S_q|}{q}\tend{q}{+\infty}\frac{1}{2}.
\end{align}
Furthermore, for any $\delta>0$, for any mesurable set $E_\delta \subset (-\delta,+\delta)^{c}$, it easy to see that
\begin{align}\label{LC1}
\int_{E_\delta} |P_q(\theta)|^\alpha d\theta \tend{q}{+\infty} |E_\delta|, 
\end{align}
where $|E_\delta|$ is the Lebesgue measure of $E_\delta$. But, it is well known that 
\begin{align}\label{LC2}
\int_{(-\delta,+\delta)^{c}}\frac{1}{q^\frac{\alpha}{2}}|D_q(e^{i x)}|^\alpha dx
\leq \frac{1}{\sqrt{q}} \frac{2^\alpha}{\big(\sin(\delta/2)\big)^\alpha} dx
 \tend{q}{+\infty}0.
\end{align}
Now, applying the triangle inequalities, we obtain
\begin{align}\label{LC3}
\Big|\Big(\int_{E_\delta} \big|\frac{2}{\sqrt{q}}.Q_q(x)\big|^\alpha dx\Big)^{\frac{1}{\alpha}}-\Big(\int_{E_\delta}\big|\frac{1}{\sqrt{q}}D_q(e^{i x)}\big|^\alpha dx\Big)^{\frac{1}{\alpha}}\Big|\\
\leq \Big\|\1_{E_\delta}.P_q(x)\Big\|_\alpha \leq 
\Big(\int_{E_\delta} \big|\frac{2}{\sqrt{q}}.Q_q(x)\big|^\alpha dx\Big)^{\frac{1}{\alpha}}+\Big(\int_{E_\delta}\big|\frac{1}{\sqrt{q}}D_q(e^{i x)}\big|^\alpha dx\Big)^{\frac{1}{\alpha}}
\end{align}
Letting $q$ goes to infinty, we see  
\begin{align}\label{LC4}
\int_{E_\delta} \Big|\frac{2}{\sqrt{q}}.Q_q(x)\Big|^\alpha dx \tend{q}{+\infty}
|E_\delta|.
\end{align}
The same reasonning yields 
$$\frac{\big\|\frac{2}{\sqrt{q}}.Q_q\big\|_\alpha}{\big\|\frac{1}{\sqrt{q}}.D_q\|_\alpha}= \frac{\big\|2.Q_q\big\|_\alpha}{\big\|D_q\|_\alpha}\tend{q}{+\infty} 1.$$

\noindent{}From this combined with \eqref{LC4} and Lemma \ref{BZ}, we obtain
\begin{eqnarray}\label{full1}
\frac{\ds \int_{E_\delta} \Big|\frac{2}{\sqrt{q}}.Q_q(x)\Big|^\alpha dx }{\big\|\frac{2}{\sqrt{q}}.Q_q\|_\alpha^\alpha}\tend{q}{+\infty}0.
\end{eqnarray}
Whence
\begin{eqnarray}\label{full2}
\frac{\ds  \int \Big|\frac{2}{\sqrt{q}}.Q_q(x)\Big|^\alpha- \int_{(-\delta,\delta)} \Big|\frac{2}{\sqrt{q}}.Q_q(x)\Big|^\alpha dx }{\big\|\frac{2}{\sqrt{q}}.Q_q\|_\alpha^\alpha}\tend{q}{+\infty}0.
\end{eqnarray}
In other words, we can write
\begin{eqnarray}\label{full3}
\frac{\ds \int_{(-\delta,\delta)} \Big|\frac{2}{\sqrt{q}}.Q_q(x)\Big|^\alpha dx }{\big\|\frac{2}{\sqrt{q}}.Q_q\|_\alpha^\alpha}\tend{q}{+\infty}1.
\end{eqnarray}
Now, by taking $\alpha=2p$, $p>1$ a positive integer and applying  Lemma \ref{Bo-Re}, we get a contradiction. 
\end{proof}
\noindent{}We conjecture the following 
\begin{conj} There is no $L^\alpha$-flat polynomails from the class of Erd\H{o}s-Littlewood polynomials (class $\Lit$) for any $\alpha <4$.
\end{conj}

\begin{thank}The author wishes to express his thanks to Jean-Marie Strelcyn for bring to his attention that the spectral Banach problem in ergodic theory can be found in Ulam's book \cite{Ulam}. 
The author would like to express their heartfelt thanks to Michael Lin for the invitation to the University of Ben-Gurion, where this work was revisited. 
\end{thank}

\end{document}